\documentclass[12pt]{amsart}
\usepackage[utf8]{inputenc}
\usepackage{amssymb}
\usepackage{geometry}
\usepackage{amsmath,amscd}
\usepackage{amsthm}
\usepackage{graphicx}
\usepackage{color,hyperref}
\usepackage{tikz}
\usepackage{tikz-cd}
\usepackage{amsfonts}
\usetikzlibrary{matrix,arrows}
\usepackage{mathrsfs}
\usepackage{amssymb}
\usepackage{adjustbox}
\begin{document}

	\newcommand{\mf}{\mathfrak}
	\newcommand{\mc}{\mathcal}
	\newcommand{\mb}{\mathbf}
	
	\newcommand{\R}{\mathbf R}
	\newcommand{\C}{\mathbf C}
	\newcommand{\Q}{\mathbf Q}
	\newcommand{\Z}{\mathbf Z}
	\newcommand{\F}{\mathbf F}
	\newcommand{\N}{\mathbf N}

	\newcommand{\Fix}{\textnormal{Fix}}
	\newcommand{\End}{\textnormal{End}}
	\newcommand{\Frob}{\textnormal{Frob}}

	\newcommand{\sm}[4]{{
			\left(\begin{smallmatrix}{#1}&{#2}\\{#3}&{#4}\end{smallmatrix}\right)}}
	\newcommand{\sv}[2]{
		\genfrac(){0pt}{1}{#1}{#2}}

	\numberwithin{equation}{section}
	
	\theoremstyle{plain}
	\newtheorem{theorem}{Theorem}
	\newtheorem{lemma}[theorem]{Lemma}
	\newtheorem{proposition}[theorem]{Proposition}
	\newtheorem*{theorem1}{Theorem 1}
	\newtheorem{corollary}[theorem]{Corollary}
		
	\theoremstyle{definition}
	\newtheorem{definition}[theorem]{Definition}
	\newtheorem{example}[theorem]{Example}
	\newtheorem{conjecture}[theorem]{Conjecture}
	\newtheorem{remark}[theorem]{Remark}

	\title{On the bmy inequality on surfaces}
	\author{SADIK TERZ\.{I}}
	
	\address{Middle East Technical University, Mathematics Department, 06800 Ankara,Turkey.}
	\email{sterzi@metu.edu.tr}
	
	
	
	
	\maketitle
	
	\begin{abstract}
	In this paper, we are concerned with the relation between the ordinarity of surfaces of general type and the failure of the BMY inequality in positive characteristic. We  consider semistable fibrations $\pi:S \longrightarrow C$ where $S$ is a smooth projective surface and $C$ is a smooth projective curve. Using the exact sequence relating the locally exact differential forms on $S$, $C$, and $S/C$, we prove an inequality relating $c_1^2$ and $c_2$ for ordinary surfaces which admit  generically ordinary semistable fibrations. This inequality differs from the BMY inequality by a correcting term which vanishes if the fibration is ordinary.
\end{abstract}

\vskip 0.5cm
\section{Introduction}

Let $S$ be a smooth projective  surface of general type over an algebraically closed field $k$. If $k$ is of characteristic zero, the Bogomolov-Miyaoka-Yau (BMY) inequality states that
$c_1^2 \le 3c_2.$
Over fields of characteristic $p > 0 $, there exist surfaces of general type which violate this inequality. This phenomenon gives rise to two intriguing problems in positive characteristic.
\begin{itemize}
	\item[$\mathbf{1)}$] To determine the precise conditions under which the BMY inequality holds.
	\item[$\mathbf{2)}$] To find relations of BMY type between the invariants $c_1^2, c_2$  of surfaces of general type under varying hypotheses.
\end{itemize}

It was conjectured that BMY should hold if the Picard scheme of the given surface is smooth (\cite{P}). This conjecture was disproved by Jang  (\cite{J1}) by constructing counter examples with smooth Picard schemes. In \cite {U}, Urzua constructs surfaces of general type (which are étale simply connected) with smooth Picard schemes for which
$c_1^2/c_2$ is dense in $[2, \infty)$. 
Recently, K. Joshi obtained inequalities relating $c_1^2, c_2$ for surfaces of general type under additional hypotheses. A notable relation is the inequality $c_1^2 \le 5c_2$ for minimal surfaces of general type which are Hodge-Witt and satisfy certain extra  hypotheses (\cite{K}, Thm. 4.17). \\

Our main result (Theorem 1) concerns ordinary smooth projective surfaces of general type which admit generically ordinary semistable fibrations. For such a surface we prove an inequality relating $c_1^2 ~\mbox{and} ~c_2$, which
differs from the BMY inequality by an additive term that vanishes if the fibration is ordinary. \\

In the second section of the paper, we collect definitions and  some basic facts about: 

\begin{itemize} 
	\item semistable fibrations of surfaces and their invariants
	\item absolute, relative, and arithmetic Frobenius morphisms
	\item locally exact differential forms on varieties
	\item ordinarity of varieties, relative ordinarity for surfaces fibered over curves and the notion of being generically ordinary
\end{itemize}
In the third section, our main aim is to verify the exactness of the following sequence relating the sheaves of locally exact differentials for a given semistable fibration $\pi:S\longrightarrow C$ 
$$0\longrightarrow \pi^*B^1_C \longrightarrow B^1_S \longrightarrow W_*B^1_{S/C} \longrightarrow0,$$
where $W$ is the arithmetic Frobenius morphism (Definition \ref{af}). This exact sequence is the main tool in proving our main result in Section 4. \\
\begin{theorem} \label{t1}
	Let $S$ be an ordinary smooth projective surface which admits a generically ordinary semistable fibration $\pi: S\longrightarrow C$ of genus $g\geq 2$ over a smooth projective curve of genus $q\geq 1$. Then the invariants $c_1^2 ~\mbox{and} ~c_2$ satisfy the following equation
	$$c^2_1= 2c_2 + \frac{12}{p-1}.h^1(B^1_{S/C})-3\delta$$
	where $\delta$ is the total number of singular points in the fibers of $\pi$.
\end{theorem} 	
\begin{corollary} \label{cor1} Under the hypotheses of Theorem 1, the following inequality holds 
	$$c^2_1 \leq  3c_2 + \frac{12}{p-1}.h^1(B^1_{S/C})-4\delta.$$
\end{corollary}

Our notation is as follows: \\

Our base field is $k = \overline{ F}_p$  for some prime $p > 0$. 

$S$ is a projective smooth surface of general type over $k$. 

$C$ is a projective smooth curve of genus $q\geq 1$. 

$\pi: S\longrightarrow C$ is a semistable fibration of genus $g \geq 2$. 

$\omega_{S/C}$ is the relative canonical bundle.

$\mathfrak{M}_g$ is the moduli space of smooth genus $g$ curves.

$\overline{\mathfrak{M}}_g$ is the moduli space of stable genus $g$ curves.

$\mathfrak{A}_{g,1,n}$ is the moduli space of principally polarized abelian varieties of dimension $g$ with a symplectic level-$n$-structure.

$\mathfrak{A}^*_{g,1,n}$ is the Satake compactification of 
$\mathfrak{A}_{g,1,n}$.

$F_X$ and $F$ are absolute and relative Frobenius morphisms, respectively, for a variety $X$.

$F$ also denotes Frobenius morphism on the cohomology groups of a variety $X$.

$B^1_X$ is the locally exact differential forms for a variety $X$.

$B^1_{X/Y}$ is the relative locally exact differential forms for a morphism $f:X \longrightarrow Y$.

\section{Basics}
In this section, we state basic facts about semistable fibrations and ordinarity of varieties (especially of surfaces). We also include basic 
cohomological tools that we shall use in the rest of the paper. \\

\vskip 0.5cm

1. {\bf Semistable fibrations} 

\begin{definition} \label{s1}
	Let $C$ be a projective curve over an algebraically closed field $k$. We say that $C$ is $\it{stable}$ (resp., $\it{semistable})$ if:
	\begin{itemize}
		\item[$\mathbf{a)}$] $C$ is connected and reduced,
		\item[$\mathbf{b)}$] all singular points are normal crossings, and
		\item[$\mathbf{c)}$] an irreducible component, isomorphic to $\mathbf{P}^1$, meets the other components in at least three (resp., two) points.
	\end{itemize}
\end{definition}
The relative version is given in the following definition.
\begin{definition}
	A proper flat morphism $f:X\longrightarrow Y$ of relative dimension one of schemes is said to be  a $\it{(semi)stable}$ curve if every geometric fiber of $f$ is a (semi)stable curve. 
\end{definition}

From now on, we are concerned with  semistable fibrations $\pi:S \longrightarrow C$ of genus $g\geq 2$ on smooth projective surfaces $S$ where the base curve $C$ is smooth of genus $q\geq 1$. Let $T \subset C$ be the set of points over which the fiber is not smooth and $t$ be the cardinality of $T$.

\begin{definition} 
	A semistable fibration $\pi:S \longrightarrow C$ is called $\it{isotrivial}$ if there exists a finite morphism $\phi:C' \longrightarrow C$ such that the fiber product $S \times_C C'$ is, birationally on $C'$, isomorphic to the trivial fibration. In this case, one can assume that $\phi$ is étale on $C-T$.
\end{definition}
\begin{remark}\label{r9} For a given semistable fibration $\pi:S \longrightarrow C$, the following hold:
	\begin{itemize}
		\item[$\mathbf{1)}$]  We have a morphism $\alpha_{\pi}:C \longrightarrow \overline{\mathfrak{M}}_g$ and $\alpha_{\pi}$ is constant iff $\pi$ is isotrivial.
		\item[$\mathbf{2)}$] The uniqueness of the semistable model implies that if the fibration $\pi$ is isotrivial, then it is smooth.
	\end{itemize}
	\end{remark}
	Next we discuss some basic invariants of semistable fibrations. For a semistable fibration 
$\pi : S \to C$, we define the following invariants:
\begin{itemize}
	\item[$\mathbf{1)}$] $d=\text{degree}(\pi_* \omega_{S/C})$,
	\item[$\mathbf{2)}$] $\delta = \displaystyle\sum_{P\in T}^{} \delta_P$ where $\delta_P$ is the number of singular points in a fiber,
	\item[$\mathbf{3)}$] $c_1$ and $c_2$ are the first and the second Chern classes of $S$,
	\item[$\mathbf{4)}$] $N = \displaystyle{\bigoplus_{P:\text{sing in a fiber}}^{}{i_P}_*(k(P))}.$ 
\end{itemize}    
These invariants satisfy the following relations
\cite[Sections 0,1, p.46-49]{Sz}.
\begin{itemize}
	\item[$\mathbf{i)}$] $\chi(\mathcal{O}_S)= \dfrac{c_1^2+c_2}{12}= \chi (\pi_*\mathcal{O}_S)-\chi (R^1\pi_* \mathcal{O}_S)$,
	\item[$\mathbf{ii)}$] $c^2_1= 12d- \delta + 8(g-1)(q-1)$,
	\item[$\mathbf{iii)}$] $c_2=\delta + 4(g-1)(q-1)$. 
\end{itemize}
From the equations in $\mathbf{ii)}$ and $\mathbf{iii)}$, we obtain the equality $c^2_1= 2c_2 + 12d- 3\delta $. \\
We also have the following exact sequences (loc. cit.).
\begin{itemize}
	\item[$\mathbf{5)}$] $0\longrightarrow \pi^*\Omega^1_C
	\longrightarrow \Omega^1_{S/k} \longrightarrow \Omega^1_{S/C}  \longrightarrow 0$
	\item[$\mathbf{6)}$] $0\longrightarrow \Omega^1_{S/C} \longrightarrow \omega^1_{S/C} \longrightarrow N  \longrightarrow 0$
\end{itemize}

\vskip 0.5cm

2. {\bf Ordinarity} \\

In this subsection, we first consider the absolute, the arithmetic and the relative Frobenius morphisms  for a morphism $f : X \longrightarrow Y$ of varieties in characteristic $p > 0$.
\begin{definition}
	The $\it{absolute ~Frobenius ~morphism}$ $F_Y : Y \longrightarrow Y$ is given by the identity map on the underlying topological space and the $p$-th power map on the structure sheaf.    
\end{definition}

\begin{remark}\label{cd}
	The following diagram is commutative:
	\begin{center}
		\begin{tikzcd}
		X \arrow[r, "F_X"] \arrow[d, "f"]
		& X \arrow[d, "f"] \\
		Y \arrow[r, "F_Y"]
		& Y
		\end{tikzcd}
	\end{center}
\end{remark}

\begin{definition} \label{af}
	The $\it{arithmetic ~Frobenius ~morphism}$ is the morphism 
	$$W : X^{(p)} = X \times_{(Y,F_Y)} Y \longrightarrow X$$ 
	obtained from $F_Y$  by the base extension.
\end{definition}

It follows from Remark \ref{cd} that there exists a unique morphism   $F : X \longrightarrow X^{(p)}$ over $Y$ fitting into the following commutative diagram.
\begin{center}
	\begin{tikzcd}
	X
	\arrow[drr, bend left, "F_X"]
	\arrow[ddr, bend right, "f"]
	\arrow[dr, dotted, "{F}" description] & & \\
	& X^{(p)} \arrow[r, "W"] \arrow[d, "f^{(p)}"]
	& X \arrow[d, "f"] \\
	& Y \arrow[r, "F_Y"]
	& Y
	\end{tikzcd}
\end{center}
$$\text{Diagram 1}$$

\begin{definition}
	$F:X \longrightarrow X^{(p)}$ is called the $\it{relative~ Frobenius ~morphism}$.   
\end{definition}

We next discuss the concepts of ordinarity and generic ordinarity. For a variety $X$ of dimension $n$, we have a complex
$$0\longrightarrow {F_X}_*\mathcal{O}_X \xrightarrow{d^1}{F_X}_* \Omega^1_{X/k} \xrightarrow{d^2}{F_X}_* \Omega^2_{X/k} \xrightarrow{d^3} \cdots \xrightarrow{d^n} {F_X}_* \Omega^n_{X/k} \longrightarrow 0.$$
The sheaf $\operatorname{Im}(d^i)$ is called the sheaf of locally exact $i$-th differential forms and is denoted by $B^i_{X/k}$.
Notice that $B^1_{X/k}= \operatorname{Coker}(F_X)$ and thus sits in the short exact sequence
\begin{equation}
  0\longrightarrow \mathcal{O}_X \xrightarrow{F_X} {F_X}_*\mathcal{O}_X \longrightarrow{}B^1_{X/k}\longrightarrow 0.\tag{1}
\end{equation}
\begin{definition}(\cite{I1}, Def. 1.1)
	We say that  $X$ is $\it{ordinary}$ if $H^i(X,B^j_{X/k})=0$ for all $i$ and $j$.
\end{definition}

We recall the characterization of ordinary curves and surfaces in terms of the action of Frobenius on the cohomology.

\begin{enumerate}
	\item[$\mathbf{1)}$]  Let $X$ be a smooth projective curve of genus $g$. Set $V=H^1(X,\mathcal{O}_X)$. \\ Let $V_n = \{ \xi \in V \mid F^m(\xi)=0 \text{ for some positive integer m}\}$ be the subspace on which $F$ is nilpotent and $V_s$ the complement of $V_n$ in $V$. In fact, $$V_s= \text{Span}(\{ \xi \in V \mid F(\xi) = \xi \} ).$$ 
	
	The natural number $\sigma_X = \operatorname{dim}_k(V_s)$ is called the $p$-rank of $X$. The following facts are well-known:
	\begin{itemize}
		\item[$\mathbf{a)}$] $X$ is an ordinary curve if and only if  $\sigma_X = g$.
		\item[$\mathbf{b)}$] The $p$-rank of a curve coincides with the $p$-rank
		$\sigma_J$ of its Jacobian. 
		\item[$\mathbf{c)}$] 
		For a semistable curve $X$, the Jacobian sits in  an extension of group schemes
		$$0 \longrightarrow G_m^s  \longrightarrow J_X \longrightarrow A \longrightarrow 0$$
		where $A$ is an abelian variety. We define the $p$-rank of $X$ by setting
		$$\sigma_X = s + \sigma_A.$$
	\end{itemize}
	
	\item[$\mathbf{2)}$] Let $S$ be a surface. By using the short exact sequence $(1)$ and Serre duality, we see that $S$ is ordinary if and only if $H^i(X,B^1_{X/k})=0$ for all $i$. This condition is equivalent to requiring
	$$F : H^i(S, \mathcal{O}_S) \to H^i(S, \mathcal{O}_S)$$
	be bijective for $i = 1, 2$.
	
\end{enumerate}

\begin{definition}(\cite{J1}, Def.2.8) 
 Let $\pi:S\longrightarrow C$ be a proper semistable fibration. We say that $\pi$ is
$\it{generically ~ordinary}$ if at least one closed fiber of $\pi$ is ordinary. (Hence almost all
closed fibers of $\pi$ are ordinary.)
\end{definition}

Next we shall prove Lemma \ref{l21} which relates the 
ordinarity of the fibers in a semistable fibration on a surface $S$ and the BMY inequality on $S$. Recall that
the BMY inequality is the relation $$c_1^2(X)\leq 3c_2(X)$$ between the Chern classes of the smooth projective surface $X$. \\

\begin{remark} \label{r13}
	We can verify the BMY inequality on certain minimal surfaces of general type by computing directly $c_1^2$ and $c_2$. For example:
	\begin{itemize}
		\item[$\mathbf{1)}$]Let $\pi:X\longrightarrow C$ be any smooth isotrivial fibration  with $g(C)\geq 2$ and $g(F)\geq 2$. Then we have 
		$$c_1^2(X)=8(g(C)-1)(g(F)-1)$$
		$$c_2(X)=4(g(C)-1)(g(F)-1)$$
		and so $c_1^2(X)=2c_2(X)$.
		\item[$\mathbf{2)}$] Let $X$ be any complete intersection smooth surface of general type of degree $d=(d_1,d_2,\cdots,d_{n-2})$ in $\mathbf{P}^n$. We assume that $\displaystyle\sum_i d_i > n+1$ and we set $\xi = \mathcal{O}_X(1).$ Then we have 
		$$c_1^2(X)= \left[ \displaystyle-n-1+ \sum_i d_i\right]^2\xi^2,$$
		$$c_2(X)= \left[ \displaystyle\frac{(n+1)n}{2}-(n+1)\sum_i d_i + \sum_{i\leq j} d_id_j  \right] \xi^2 $$
		and so $c_1^2(X)\leq c_2(X)$.\\
In particular, the BMY inequality holds on any smooth surface of degree $d\geq 5$ in $\mathbf{P}^3.$
		
	\end{itemize}
\end{remark}

The following lemma provides other examples for which the BMY inequality holds. In the proof of 
this lemma, we will use the basic properties of the Ekedahl-Oort strata in $\mathfrak{A}^*_{g,1,n}$ (\cite{O}, Sections 2-6). 
\begin{lemma}\label{l21}
	Let $\pi:S\longrightarrow C$ be a semistable fibration of genus $g$ such that the fibers are all
	(i) ordinary, or (ii) of $p$-rank $g-1$, or
	(iii) supersingular but not superspecial, or (iv) superspecial. Then $\pi$
	is isotrivial. Hence, the BMY inequality
	holds on the surface $S$.
\end{lemma}
\begin{proof} 
	Let $U \subseteq S$ be the union of the smooth fibers of $\pi$ and set $C' = \pi(U)$. The canonical principal polarization on the relative Jacobian $J_{U/C'}$ extends to a principal cubic structure on $J_{S/C}$ (\cite{LMB}, Chap. II, Thm. 3.5) which we denote by $\Theta$. Thus, we obtain a morphism
	$$h:C \longrightarrow \mathfrak{A}^*_{g,1,n}$$  given by $$p \mapsto (J_{S_p}, \Theta_p).$$ 
	In each of the cases $\mathbf{(i)}$, $\mathbf{(ii)}$, $\mathbf{(iii)}$  and $\mathbf{(iv)}$ we verify that the image $h(C)$ lies in an  Ekedahl-Oort stratum determined by a unique elementary sequence $\varphi$ (\cite{O}, Definition 2.1.).
	\begin{itemize}
		\item[$\mathbf{Case ~(i)}$] : $|\varphi|=g\cdot(g+1)/2$ correspondes only to the elementary sequence $$\varphi = \{1,2,\cdots,g\}.$$
		\item[$\mathbf{Case ~(ii)}$] : $|\varphi|=g\cdot(g+1)/2-1$ is obtained only from the elementary sequence 
		$$\varphi = \{1,2,\cdots,g-1,g-1\}.$$
		\item[$\mathbf{Case ~(iii)}$] : If the fibers are supersingular, but not superspecial, then we have $|\varphi|=1$ for which the elementary sequence is
		$$\varphi = \{0,0,\cdots,0,1\}.$$
		\item[$\mathbf{Case ~(iv)}$] : If the fibers are superspecial, then we have $|\varphi|=0$ which arises only from the elementary sequence $$\varphi = \{0,0,\cdots,0\}.$$
	
	\end{itemize}
It follows from (\cite{O}, Thm.6.4) that the image $h(C)$ is quasi-affine, and so the morphism $h$ is constant on $C$. Then we see that the smooth fibers of $\pi$ are all isomorphic by applying the Torelli theorem. That is, $\pi : S \longrightarrow C$ is an isotrivial semistable fibration. Therefore, $\pi$ is smooth and  
	$d = 0, ~\delta = 0$. It follows that $c_1^2=2c_2$ by Remark \ref{r13}.
	
\end{proof}

\section{Some Local Algebra}
Let $\pi: S\longrightarrow C$ be a fibration as in Section 1. Then one has the following short exact sequence of $\mathcal{O}_S$-modules
 \begin{equation}
  0\longrightarrow \pi^*\Omega^1_{C/k} \longrightarrow \Omega^1_{S/k} \longrightarrow \Omega^1_{S/C}  \longrightarrow 0 . \tag{2}
\end{equation}
Since $F_S$ is a finite morphism, we obtain the following exact sequence
\begin{equation}
0\longrightarrow {F_S}_*\pi^*\Omega^1_{C/k} \longrightarrow {F_S}_*\Omega^1_{S/k} \longrightarrow 
{F_S}_*\Omega^1_{S/C} \longrightarrow 0  \tag{3}
\end{equation}
applying ${F_S}_*$ in (2).\\\\
In Lemma \ref{l2}, we prove the existence of a short exact sequence of sheaves of locally exact differential forms (analogous to (3)).
To this end, we first recall the definition of the relative locally exact differential forms $B^1_{S/C}$. Using the relative deRham complex 
\begin{center}
	\begin{tikzcd}
	0 \arrow[r, ""] & \mathcal{O}_S \arrow[r, "d"] & \Omega^1_{S/C}
	\arrow[r," "] & 0
	\end{tikzcd}
\end{center}
and the relative Frobenius map  $F:S\longrightarrow S^{(p)}$, we obtain a complex
\begin{center}
	\begin{tikzcd}
	0 \arrow[r, ""] & F_*\mathcal{O}_S \arrow[r, "F_*d"] & F_*\Omega^1_{S/C}
	\arrow[r," "] & 0.
	\end{tikzcd}
\end{center}

\begin{definition}
The sheaf $B^1_{S/C}:= \operatorname{Im}(F_*d)$ is called the $\it{relative ~locally ~exact ~differential ~forms}$ for the fibered surface $\pi:S\longrightarrow C$.    
\end{definition}
\begin{remark}
	$B^1_{S/C}$ sits in the following short exact sequence 
	$$ 0 \longrightarrow \mathcal{O}_{S^{(p)}} \longrightarrow F_*\mathcal{O}_S \longrightarrow B^1_{S/C} \longrightarrow 0. $$
\end{remark}
\begin{remark}
	Let $X$ be a scheme over a perfect field of positive characteristic $p$ and $W$ be the arithmetic Frobenius morphism. Then $X^{(p)} \simeq X$ via the map $W$, because $F_k:\text{Spec}(k) \longrightarrow \text{Spec}(k)$ is an isomorphism. Therefore, $B^1_{X/k}$ can be viewed as a sheaf of $\mathcal{O}_X$-modules on $X$. 
\end{remark}

\begin{proposition}
	Let $\pi:S\longrightarrow C$ be a fibration. Then there exists a morphism of $\mathcal{O}_S$-modules $\phi:\pi^* {F_C}_*\Omega^1_{C/k}\longrightarrow {F_S}_*\pi^* \Omega^1_{C/k}$.
\end{proposition}
\begin{proof}
	Since the statement is local on the both source and target, we may assume that $S=\operatorname{Spec}(R)$ and $C=\operatorname{Spec}(A)$. Therefore, $\pi^* {F_C}_*\Omega^1_{C/k}= \widetilde{{F_A}_*\Omega_{A/k} \bigotimes_A R}$ and 
	${F_S}_*\pi^* \Omega^1_{C/k} = \widetilde{ {F_R}_* ( \Omega^1_{C/k} \bigotimes_A R)}$ where 
	$F_A$ and $F_R$ denote the relevant Frobenius morphisms.
	Define an $R$-module homomorphism $$\varphi:{F_A}_*\Omega_{A/k} \otimes_A R \longrightarrow {F_R}_* ( \Omega^1_{C/k} \otimes_A R) $$ $$\displaystyle\sum_{i=1}^{n} a^i_1.a^i_2da^i_3 \otimes r_i \longmapsto \displaystyle\sum_{i=1}^{n} (a^i_1)^pa^i_2da^i_3 \otimes r^p_i.$$
	$\varphi $ is a well-defined $R$-module homomorphism due to the equality $$\displaystyle\sum_{i=1}^{n} a^i_1.a^i_2da^i_3 \otimes r^p_i = \displaystyle\sum_{i=1}^{n} r_i.((a^i_1)^pa^i_2da^i_3 \otimes 1). $$ \\
	Hence, we have a morphism of $\mathcal{O}_S$-modules 
	$$\widetilde{\varphi}:\widetilde{{F_A}_*\Omega_{A/k} \otimes_A R} \longrightarrow \widetilde{ {F_R}_* ( \Omega^1_{C/k} \otimes_A R)}.$$
	Thus, we obtain the required morphism of $\mathcal{O}_S$-modules
	$$\phi:\pi^* {F_C}_*\Omega^1_{C/k}\longrightarrow {F_S}_*\pi^* \Omega^1_{C/k}.$$
\end{proof}
\begin{remark} \label{r1}
	Since $B^1_{A/k}$ is a submodule of $ {F_A}_*\Omega^1_{A/k}$, $B^1_{A/k}\otimes_A R$ is a submodule of \\
	$ {F_A}_*\Omega^1_{A/k}\otimes_A R$. Therefore, we get a morphism of $R$-modules 
	$$ \overline{\varphi}: B^1_{A/k}\otimes_A R \longrightarrow {F_R}_* ( \Omega^1_{C/k} \otimes_A R) $$ by restricting $\varphi$ to $B^1_{A/k}\otimes_A R$. For a given differential form  $a_1.a_2da_3 \otimes 1 \in B^1_{A/k}\otimes_A R $, there exists $a_0 \in A$ such that $a_2da_3 = da_0$ as $a_2da_3 \in B^1_{A/k} $. Hence, $a_1.a_2da_3=a_1.da_0=a^p_1da_0=da^p_1a_0$. As a result, any element $\omega \in B^1_{A/k}\otimes_A R $ can be written as $\omega = da\otimes r$ for some $a\in A$ and $r\in R$. Thus, we see that $\overline{\varphi}$ is an injective homomorphism of $R$-modules. Therefore, we get an injective morphism of $\mathcal{O}_S$-modules 
	$$\overline{\phi}: \pi^*B^1_{C/k} \hookrightarrow {F_S}_*\pi^* \Omega^1_{C/k}.$$
\end{remark}

By the exact sequence (3)  and Remark \ref{r1}, we obtain a complex of $\mathcal{O}_S$-modules  $$ 0 \longrightarrow \pi^*B^1_{C/k}  \longrightarrow {F_S}_*\Omega^1_{S/k} \longrightarrow {F_S}_*\Omega^1_{S/C}  \longrightarrow 0.$$ Restricting to the the open subscheme $\operatorname{Spec}(A) \subset C$, we obtain the following complex of $R$-modules 
\begin{center}
	\begin{tikzcd}
	0 \arrow[r, ""] & B^1_{A/k} \otimes_A R \arrow[r, "\overline{u}"] &
	{F_R}_*\Omega^1_{R/k} \arrow[r, "\overline{v}"] & {F_R}_*\Omega^1_{R/A} \arrow[r," "] & 0 
	\end{tikzcd}
\end{center}
where $\overline{u} = {F_R}_*u \circ \overline{\varphi}$ and
$\overline{v}={F_R}_*v$. \\

We will work out the details of the $R$-module structure on ${F_R}_*\Omega^1_{R/A}$.
\begin{remark}
	One has the following commutative diagram 
	(which corresponds to  Diagram 1 applied to the morphism $f : S \longrightarrow C$ over $\operatorname{Spec}(A) \subset C$).
	\begin{center}
		\begin{tikzcd}
		A  \arrow[r, "F_A"] \arrow[d, "\psi"] 
		& A \arrow[d, "\psi^{(p)}"] \arrow[ddr, bend left, "\psi "] &\\
		R \arrow[r, "W"] \arrow[drr, bend right, " F_R"]
		& R^{(p)} \arrow[dr, dotted, "{F}" description] &  \\
		& & R
		\end{tikzcd}
	\end{center}
	$$\text{Diagram 2}$$
	
	In Diagram 2, 
	\begin{center}
		\begin{itemize}
			\item[$\mathbf{a)}$] $R^{(p)}=R\otimes_{A,F_A}A$,
			\item[$\mathbf{b)}$] $\psi: A\longrightarrow R$ is a ring homomorphism which corresponds to $\pi:S\longrightarrow C$,
			\item[$\mathbf{c)}$] $F:R^{(p)} \longrightarrow R$ given by $r\otimes a \longmapsto ar^p$ corresponding to $F:S\longrightarrow S^{(p)}$,
			\item[$\mathbf{d)}$] $W:R\longrightarrow R^{(p)}$ given by $ar\longmapsto r\otimes a^p$ corresponding to $W:S^{(p)} \longrightarrow S$,
			\item[$\mathbf{e)}$] $W\circ F=F_R$.
		\end{itemize}
	\end{center}
	Therefore, ${F_R}_*\Omega^1_{R/A}= W_*(F_*\Omega^1_{R/A})$. Then for any $\omega = (a_1r_1).r_2dr_3 \in {F_R}_*\Omega^1_{R/A} $, we have 
	$$\omega = (a_1r_1).r_2dr_3 = (r_1\otimes a^p_1).r_2dr_3 \text{ via }W$$ and 
	$$(r_1\otimes a^p_1).r_2dr_3 = a^p_1r^p_1r_2dr_3 \text{ via }F$$ i.e., first we make $\Omega^1_{R/A}$ an $R^{(p)}$-module via $F$ and then via the map $W:R\longrightarrow R^{(p)}$, $\Omega^1_{R/A}$ becomes an $R$-module. Moreover, we may view $W_*B^1_{R/A}$ as a subsheaf of ${F_R}_*\Omega^1_{R/A}=W_*(F_*\Omega^1_{R/A})$ because $B^1_{R/A}$ is the subsheaf of $F_*\Omega^1_{R/A}$.\\
\end{remark}

\begin{lemma}\label{l2}
	Let $\pi:S \longrightarrow C$ be a fibration on a smooth projective surface $S$. Then one has a short exact sequence
	\begin{equation}
	    0\longrightarrow \pi ^*B^1_{C/k} \longrightarrow B^1_{S/k} \longrightarrow W_*B^1_{S/C} \longrightarrow 0. \tag{4}
	\end{equation}
\end{lemma}
\begin{proof}
	Let $\psi:A\longrightarrow R$ be the ring homomorphism corrosponding to $\pi$. Let $da \otimes r$ be in $B^1_{A/k} \otimes_A R. $ We may restrict $\overline{v}$ to the subsheaf $B^1_{R/k}$ of ${F_R}_*\Omega^1_{R/k}$ as $\overline{u}(da\otimes r)=r^pd\psi (a) = dr^p\psi (a)  \in B^1_{R/k}$. Also for a given $dr\in B^1_{R/k}$ since $\overline{v}(dr)=dr\in W_*B^1_{R/A}$, we have the following sequence of $R$-modules:
	\begin{center}
		\begin{tikzcd}
		0 \arrow[r, ""] & B^1_{A/k} \otimes_A R \arrow[r, "\overline{u}"] &
		B^1_{R/k} \arrow[r, "\overline{v}"] & W_*B^1_{R/A} \arrow[r," "] & 0.
		\end{tikzcd}
	\end{center}
	To complete the proof, we need to prove the following claims: 
	\begin{itemize}
     \item[$\mathbf{Claim ~ (1)}$]: $\overline{u}$ is injective,
    \item[$\mathbf{Claim ~ (2)}$]: $\operatorname{Im}(\overline{u})=\operatorname{Ker}(\overline{v})$,
    \item[$\mathbf{Claim ~ (3)}$]: $\overline{v}$ is surjective.
	\end{itemize}
	The first claim follows from Remark \ref{r1}. \\
	
	Clearly, $\operatorname{Im}(\overline{u})\subseteq\operatorname{Ker}(\overline{v})$ by the short exact sequence (2). Let $dr$ be in $\operatorname{Ker}(\overline{v})$. We have $dr= 0 $ in $B^1_{R/A}$ which implies that $r\in A$ i.e., there exists $a\in A$ such that $r=\psi (a)$. Therefore, $dr= d\psi (a) = \overline{u}(da\otimes 1)$. As a result, we have  $$\operatorname{Ker}(\overline{v})=\operatorname{Im}(\overline{u}) $$
	which completes the proof of the second claim. \\
	
	Let $[r_1.(r_2\otimes a)].dr_3$ be in $W_*B^1_{R/A}$. Then the last claim follows by the equality
	$$[r_1.(r_2\otimes a)].dr_3 = [(r_1\otimes 1)(r_2 \otimes a)].dr_3= (r_1r_2 \otimes a).dr_3= ar_1^pr_2^pdr_3=d\psi (a)r_1^pr_2^pr_3 = \overline{v}(ar_1^pr_2^pr_3).$$ 
	
	Thus, the sequence
	\begin{center}
		\begin{tikzcd}
		0 \arrow[r, ""] & B^1_{A/k} \otimes_A R \arrow[r, "\overline{u}"] &
		B^1_{R/k} \arrow[r, "\overline{v}"] & W_*B^1_{R/A} \arrow[r," "] & 0
		\end{tikzcd}
	\end{center}
	is a short exact sequence of $R$-modules. This implies the following is a short exact sequence of $\mathcal{O}_S$-modules
	\begin{center}
		\begin{tikzcd}
		0 \arrow[r, ""] &  \widetilde{B^1_{A/k} \otimes_A R }\arrow[r, ""] & \widetilde{B^1_{R/k} }\arrow[r, ""] & \widetilde{ W_*B^1_{R/A} } \arrow[r," "] & 0.
		\end{tikzcd}
	\end{center}
	Therefore, we have the following short exact sequence of $\mathcal{O}_S$-modules 
	$$0\longrightarrow \pi ^*B^1_{C/k} \longrightarrow B^1_{S/k} \longrightarrow W_*B^1_{S/C} \longrightarrow 0 .$$
\end{proof}

\section{The Main Result}
In this section, we prove our main result, Theorem \ref{t1}. The main ingredient is the short exact sequence of locally exact differential forms constructed in the preceding section. We will use the following well-known result \cite[Chapter 3, Exercises 8.3]{H} in the proof of Theorem \ref{t1}. \\

\begin{proposition}\label{p2} 
	Let $f: X \longrightarrow Y$ be a morphism of ringed spaces, let $\mathcal{F}$ be an $\mathcal{O}_X$-module and let $\mathcal{E}$ be a locally free $\mathcal{O}_Y$-module of finite rank. Then
	$$ R^if_*(\mathcal{F}\otimes f^*\mathcal{E})=R^if_*(\mathcal{F})\otimes \mathcal{E} $$
	$ \text{for all} ~~~i \ge 0.$
\end{proposition}

\begin{remark}
	Let $\pi:S\longrightarrow C$ be a semistable fibration as in the statement of \ref{t1}. Then $ B^1_{C/k}$ is a locally free $\mathcal{O}_C$-module of rank $p-1$ and of degree $(p-1)(q-1)$. By Proposition \ref{p2}, we have   $$R^i\pi_* \pi^*B^1_{C/k}= R^i\pi_*(\mathcal{O}_S \otimes \pi^*B^1_{C/k})= R^i\pi_* \mathcal{O}_S \otimes B^1_{C/k}.$$
	
	We calculate the rank and the degree of the sheaf  $\mathcal{M}
	= R^1\pi_* \mathcal{O}_S \otimes B^1_{C/k}$
	and we get 
	$$\text{rank}(\mathcal{M})=r(p-1) \text{ and } \text{deg}(\mathcal{M})=(p-1)e+r(p-1)(q-1)$$
	where $r=\text{rank}(R^1\pi_*\mathcal{O}_S) \text{ and }e=\text{deg}(R^1\pi_*\mathcal{O}_S)$.
\end{remark}

Consider the Leray spectral sequence attached to the sheaf $\mathcal{F}= \pi^*(B^1_{C/k})$ on the fibration $\pi:S\longrightarrow C$, namely
\begin{equation}
	 E_2^{pq}=H^p(C,R^q\pi_* \pi^*B^1_{C/k})=H^p(C,R^q\pi_* \mathcal{O}_S \otimes B^1_{C/k}) \implies H^{p+q}(S,\pi^*B^1_{C/k}).  \tag{$\ast$}
	\end{equation}
We have the following properties:
\begin{itemize}
	\item[$\mathbf{a)}$] Since $C$ is a curve, $H^p(C,-)=0$ for $p>1$.
	\item[$\mathbf{b)}$] By Corollary 11.2 in \cite[Chapter 3]{H}, $R^q\pi_* \mathcal{O}_S = 0$ for $q>1.$
	\item[$\mathbf{c)}$] By proposition \ref{p2} and since $\pi_*\mathcal{O}_S = \mathcal{O}_C$, we have $\pi_* \pi^*B^1_{C/k}= B^1_{C/k}$.
\end{itemize}
Therefore, we get
$$H^{0}(S,\pi^*B^1_{C/k}) = H^0(C,B^1_{C/k}) \text{ and }H^{2}(S,\pi^*B^1_{C/k}) = H^1(C,R^1\pi_* \mathcal{O}_S \otimes B^1_{C/k}).$$
If we assume that $C$ is an ordinary curve, then we also have $$H^{1}(S,\pi^*B^1_{C/k}) = H^0(C,R^1\pi_* \mathcal{O}_S \otimes B^1_{C/k}). $$

\vskip 0.2cm

Now we prove Theorem \ref{t1}. We remark that the hypothesis in Theorem 1 differs from the hypothesis in Lemma \ref{l21}; we remove the $p$-rank condition on the non-smooth fibers, but now we assume that $S$ is ordinary. 
\vskip 0.5cm
\begin{theorem1}
	Let $S$ be an ordinary smooth projective surface which admits a generically ordinary semistable fibration $\pi: S\longrightarrow C$ of genus $g\geq 2$ over a smooth projective curve of genus $q\geq 1$. Then the invariants $c_1^2 ~\mbox{and} ~c_2$ satisfy the following equation
	$$c^2_1= 2c_2 + \frac{12}{p-1}.h^1(B^1_{S/C})-3\delta$$
	where $\delta$ is the total number of singular points in the fibers of $\pi$.
\end{theorem1}
\begin{proof}
Recall that for semistable fibrations, we have the equality:
$$c^2_1=2c_2+ 12d -3\delta.$$
We will prove that $(p-1)d = h^1(B^1_{S/C})$. For this purpose, we will use the short exact sequence (4) proved in Lemma \ref{l2} :
	$$0\longrightarrow \pi ^*B^1_{C/k} \longrightarrow B^1_{S/k} \longrightarrow W_*B^1_{S/C} \longrightarrow 0.$$ 
	We have the long exact sequence 
	\begin{center}
		$$0\longrightarrow H^0(S,\pi^*B^1_{C/k}) \longrightarrow H^0(S,B^1_{S/k}) \longrightarrow H^0(S,W_*B^1_{S/C}) \longrightarrow \cdots$$
		$$\cdots \longrightarrow H^1(S,\pi^*B^1_{C/k}) \longrightarrow H^1(S,B^1_{S/k}) \longrightarrow H^1(S,W_*B^1_{S/C}) \longrightarrow \cdots$$ 
		$$\cdots \longrightarrow H^2(S,\pi^*B^1_{C/k}) \longrightarrow H^2(S,B^1_{S/k}) \longrightarrow H^2(S,W_*B^1_{S/C}) \longrightarrow 0.$$
	\end{center}
	We note the following:
	\begin{itemize}
		\item[$\mathbf{1)}$] Since $S$ is an ordinary surface, $$H^i(S,B^1_{S/k})=0$$ 
		$ \text{for all} ~~~i \ge 0.$ It is easily concluded that $H^2(S,W_*B^1_{S/C})=0.$ 
		\item[$\mathbf{2)}$] Since $\pi $ is a generically ordinary semistable fibration,
		$\pi^{(p)}_*B^1_{S/C}\mid_U=0$ where $U$ is the ordinary locus of $\pi$.
		However, $B^1_{S/C} $ is flat over $\mathcal{O}_C$ so $\pi^{(p)}_*B^1_{S/C}=0$. It follows that $H^0(S,W_*B^1_{S/C})=H^0(S^{(p)},B^1_{S/C})=H^0(C,\pi^{(p)}_*B^1_{S/C})=0$ as $W$ is a finite morphism.
	\end{itemize}
	Therefore, by the long exact sequence 
	$$H^0(S,\pi^*B^1_{C/k})=H^1(S,\pi^*B^1_{C/k})=0$$
	and
	$$H^1(S,W_*B^1_{S/C})= H^2(S,\pi^*B^1_{C/k}).$$
	Now, recall that $S$ is assumed to be an ordinary surface. Then $C$ is an ordinary curve and hence by using the Leray spectral sequence $(\ast)$ we have an equality:
	 $$\chi(\mathcal{M})= h^0(C,\mathcal{M}) - h^1(C,\mathcal{M})=h^1(S,\pi^*B^1_{C/k})-h^2(S,\pi^*B^1_{C/k}).$$
	 On the other hand, $$\chi(\mathcal{M})=\operatorname{deg}(\mathcal{M})-(q-1)\operatorname{rank}(\mathcal{M}).$$
	Therefore, we have $$-h^2(S,\pi^*B^1_{C/k})=\chi(\mathcal{M})$$ 
	and so
	$$-h^2(S,\pi^*B^1_{C/k})=(p-1)e + r(p-1)(q-1)-(q-1)r(p-1).$$ 
	This implies that $(p-1)d = h^2(S,\pi^*B^1_{C/k})=h^1(S,W_*B^1_{S/C})=h^1(B^1_{S/C})$ where $d=-e$. Recall that since $\pi: S\longrightarrow C$ is a semistable fibration, we have $c^2_1=2c_2 + 12d -3\delta.$ Substituting $\displaystyle d = \frac{1}{p-1} \cdot h^1(B^1_{S/C}) $, we obtain the equality:
	$$c^2_1=2c_2 + \frac{12}{p-1}.h^1(B^1_{S/C})-3\delta.$$
\end{proof}

Since for the fibration in Theorem \ref{t1} 
$$c_2 - \delta = 4(g-1)(q - 1) \ge 0$$
we obtain the inequality given in Corollary \ref{cor1}, namely 
$$c^2_1 \leq  3c_2 + \frac{12}{p-1}.h^1(B^1_{S/C})-4\delta.$$

\end{document}